\documentclass[12pt]{amsart}
\usepackage{url}
\usepackage{pslatex}
\usepackage{amsmath,amsthm,amssymb,amsfonts, stmaryrd}
\usepackage{fancybox}
\usepackage{color, graphicx}
\usepackage[T1]{fontenc} 
\usepackage[usenames,dvipsnames,svgnames,table]{xcolor}
\newtheorem{theorem}{Theorem}[section]

\newtheorem{conjecture}[theorem]{Conjecture}
\newtheorem{corollary}[theorem]{Corollary}

\newtheorem{definition}[theorem]{Definition}

\newtheorem{lemma}[theorem]{Lemma}
\newtheorem{notation}[theorem]{Notation}

\newtheorem{proposition}[theorem]{Proposition}
\newtheorem{remark}[theorem]{Remark}

\DeclareMathOperator{\card}{card}
\DeclareMathOperator{\im}{Im}
\DeclareMathOperator{\depth}{depth}

\newcommand{\Z}{\mathbb Z}

\newcommand{\N}{\mathbb N}

\DeclareMathOperator{\F}{F} 
\DeclareMathOperator{\T}{T} 

\newcommand{\fc}{\mathcal F}

\newcommand{\wh}[1]{\skew{4}\widehat{\widehat{#1}}} 

\newcommand{\vs}[1]{\langle #1 \rangle}

\DeclareMathOperator{\genus}{genus}

\usepackage{tikz}

\begin{document}

\title{Gapsets and numerical semigroups}

\author{Shalom Eliahou and Jean Fromentin}

\address{Shalom Eliahou, Univ. Littoral C\^ote d'Opale, EA 2597 - LMPA - Laboratoire de Math\'ematiques Pures et Appliqu\'ees Joseph Liouville, F-62228 Calais, France and CNRS, FR 2956, France} \email{eliahou@univ-littoral.fr}

\address{Jean Fromentin, Univ. Littoral C\^ote d'Opale, EA 2597 - LMPA - Laboratoire de Math\'ematiques Pures et Appliqu\'ees Joseph Liouville, F-62228 Calais, France and CNRS, FR 2956, France} \email{fromentin@math.cnrs.fr}

\maketitle

\begin{abstract} For $g \ge 0$, let $n_g$ denote the number of numerical semigroups of genus $g$. A conjecture by Maria Bras-Amor\'os in 2008 states that the inequality $n_{g} \ge n_{g-1}+n_{g-2}$ should hold for all $g \ge 2$. Here we show that such an inequality holds for the very large subtree of numerical semigroups satisfying $c \le 3m$, where $c$ and $m$ are the conductor and multiplicity, respectively. Our proof is given in the more flexible setting of \emph{gapsets}, i.e. complements in $\N$ of numerical semigroups.
\end{abstract}

\section{Introduction}
Denote $\N=\{0,1,2,3,\dots\}$ and $\N_+=\N\setminus \{0\}=\{1,2,3,\dots\}$. For $a,b \in \Z$, let $[a,b]=\{z \in \Z \mid a \le z \le b\}$ and $[a,\infty[=\{z \in \Z \mid a \le z\}$ denote the integer intervals they span. A \emph{numerical semigroup} is a subset $S \subseteq \N$ containing $0$, stable under addition and with finite complement in $\N$. Equivalently, it is a subset $S \subseteq \N$ of the form $S = \vs{a_1,\dots,a_n}=\N a_1 + \dots + \N a_n$ for some globally coprime positive integers $a_1,\dots,a_n$.

\smallskip
For a numerical semigroup $S \subseteq \N$, its \emph{gaps} are the elements of $\N \setminus S$, its \emph{genus} is $g=|\N \setminus S|$, its \emph{multiplicity} is $m = \min S \setminus \{0\}$, its \emph{Frobenius number} is $f = \max \Z\setminus S$, its \emph{conductor} is $c=f+1$, and its \emph{embedding dimension}, usually denoted $e$, is the least number of generators of $S$, \emph{i.e.} the least $n$ such that $S = \vs{a_1,\dots,a_n}$. Note that the conductor $c$ of $S$ satisfies $c+\N \subseteq S$, and is minimal with respect to this property since $c-1=f \notin S$.

\smallskip
Given $g \ge 0$, the number $n_g$ of numerical semigroups of genus $g$ is finite, as easily seen. The values of $n_g$ for $g=0,\dots,15$ are as follows:
$$
1,1,2,4,7,12,23,39,67,118,204,343,592,1001,1693,2857.
$$ 
In 2008, Maria Bras-Amor\'os made some remarkable conjectures concerning the growth of $n_g$. In particular, she conjectured that 
\begin{equation}\label{strong conjecture}
n_g  \,\ge\, n_{g-1}+n_{g-2}
\end{equation}
for all $g \ge 2$. This conjecture is widely open. Indeed, even the weaker inequality $n_g \,\ge\,  n_{g-1}$, whose validity has been settled by Alex Zhai \cite{Z} for all sufficiently large $g$, remains to be proved for all $g \ge 1$. In that same paper, Zhai showed that `most' numerical semigroups $S$ satisfy $c \le 3m$, where $c$ and $m$ are the conductor and multiplicity of $S$, respectively. For a more precise statement, let us denote
\begin{eqnarray*}
n'_g & = & \mbox{the number of numerical semigroups } S \\  
& & \mbox{of genus $g$ satisfying } c \le 3m.
\end{eqnarray*}
Zhai showed then that $\lim_{g \to \infty} n'_g/n_g =1$, as had been earlier conjectured by Yufei Zhao \cite{Zhao}. The values of $n_g'$ for $g=0,\dots,15$ are as follows:
$$
1,1,2,4,6,11,20,33,57,99,168,287,487,824,1395,2351.
$$ 

In this paper, we show that the conjectured inequality \eqref{strong conjecture} holds for $n'_g$. Even more so, we shall prove the following bounds on $n'_g$ for all $g \ge 3$:
\begin{equation}\label{bounds}
n'_{g-1}+n'_{g-2} \,\le\, n'_{g} \,\le\, n'_{g-1}+n'_{g-2}+n'_{g-3},
\end{equation}
the strongest partial result so far towards \eqref{strong conjecture}.

\smallskip
The contents of this paper are as follows. In Section~\ref{section background}, we recall the necessary background, including the tree of numerical semigroups, and we introduce the \emph{depth}, a key parameter for numerical semigroups which had no specific name yet. In Section~\ref{section gapsets}, we introduce \emph{gapsets}, i.e. complements in $\N$ of numerical semigroups, and \emph{gapset filtrations}. We also introduce $m$-extensions and $m$-filtrations to facilitate their study. In Section~\ref{section q <= 2}, we consider the case of depth at most $2$, i.e. where $c \le 2m$. Sections~\ref{section lower} and~\ref{section upper} are the heart of the paper, where we use the setting of gapsets to establish the left and right inequalities in \eqref{bounds}, respectively. In Section~\ref{section graph}, we show that the tree of numerical semigroups may be naturally embedded in a richer graph whose new edges played a key role towards establishing \eqref{bounds}. Finally, in Section~\ref{section further} we propose some related conjectures and announce some forthcoming results on gapsets of small multiplicity. An Appendix gives the exact values of $n'_g$ for $g = 1, \dots, 60$.

\section{Background}\label{section background}
\smallskip
Numerical semigroups $S \subseteq \N$ may be defined in two equivalent yet quite distinct ways.
\begin{definition}\hspace{0mm}
\begin{enumerate}
\item As cofinite submonoids of $\N$. That is, as subsets $S \subseteq \N$ containing $0$, stable under addition and with finite complement $\N \setminus S$.
\item As subsets of $\N$ of the form $S = \vs{a_1,\dots,a_n}=\N a_1 + \dots + \N a_n$, where $a_1,\dots,a_n \in \N_+$ and $\gcd(a_1,\dots,a_n)=1$.
\end{enumerate}
\end{definition}

For most numerical semigroups, going from one description to the other one is computationally costly. That is, the description of $S$ as $S = \vs{a_1,\dots,a_n}$ does not easily yield $\N \setminus S$ - think of the Frobenius problem -  nor conversely.

\subsection{Counting numerical semigroups by genus} 

Given $g \ge 0$, the number $n_g$ of numerical semigroups of genus $g$ is finite, as easily seen. As mentioned above, the values of $n_g$ for $g=0,\dots,15$ are as follows:
$$
1,1,2,4,7,12,23,39,67,118,204,343,592,1001,1693,2857.
$$ 
In 2006, Maria Bras-Amor\'os pushed the computation of this sequence up to $g=50$ and came up with beautiful conjectures about its growth~\cite{Br08}.
\begin{conjecture} The following probably hold.
\begin{align}
& \hspace{0.1cm} n_{g} \, \ge\,  n_{g-1} + n_{g-2} \hspace{0.5cm} \textrm{ for all } g \ge 2,  \hspace{5cm}\label{c1}\\
& \lim_{g \to \infty} (n_{g-1} + n_{g-2})/n_{g}  =  1 \label{c2}, \\
& \lim_{g \to \infty} n_{g}/n_{g-1} = (1+\sqrt{5})/2. \label{c3}
\end{align}
\end{conjecture}
As \eqref{c1} is still widely open, a weaker version has been formally proposed, possibly first in \cite{K1}, even though the problem was already informally mentioned in~\cite{Zhao} for instance. 

\begin{conjecture}\label{simple growth} The inequality $n_{g} \ge n_{g-1}$ should hold for all $g \ge 1$.
\end{conjecture}
Zhai~\cite{Z} showed that $n_{g} \ge n_{g-1}$ does indeed hold for all sufficiently large $g$, but whether it holds for all $g \ge 1$ remains open at the time of writing. See \cite{K2} for a nice survey on related questions.

\subsection{The depth} 
\begin{definition}
Let $S$ be a numerical semigroup of multiplicity $m$ and conductor $c$. We call \emph{depth} of $S$ the integer $q = \lceil c/m \rceil$, and we denote it by $\depth(S)$.
\end{definition}

The only numerical semigroup of depth $0$ is $S=\N$. Since $c \ge m$ if $S \not=\N$, the numerical semigroups of depth 1 are exactly those for which $c=m$, i.e. those of the form
$$
S = \{0\} \cup [m, \infty[
$$
for some $m \ge 2$. These specific numerical semigroups are called \emph{ordinary} in the current literature, but a more appropriate and descriptive term would be \emph{superficial}.

\medskip
The depth is an important parameter of numerical semigroups, even though it wasn't specifically named before the present paper. For instance, among various partial results, Wilf's conjecture has been shown to hold for numerical semigroups of depth $q = 2$ in \cite{K1} and in the more demanding case $q =3$ in \cite{E}. Moroever, near-misses in Wilf's conjecture have been constructed for depth $q \ge 4$ and embedding dimension $3$ in \cite{D}, and for depth $q = 4$ and arbitrary large embedding dimension in \cite{EF}. Zhao showed that the number of numerical semigroups of genus $g$ and depth $q \le 2$ is equal to the Fibonacci number $\F_{g+1}$~\cite{Zhao}. More importantly for this paper, Zhao conjectured in~\cite{Zhao}, and Zhai proved in~\cite{Z}, that `most' numerical semigroups are of depth $q \le 3$. More precisely, that \emph{among all numerical semigroups of genus $g$, the proportion of those of depth $q \le 3$ tends to $1$ as $g$ tends to infinity}. This phenomenon is illustrated in Figure~\ref{F:Cloche} below.

\subsection{The tree of numerical semigroups}
The set of all numerical semigroups may be organized into a tree $\mathcal{T}$, rooted at $\N=\vs{1}$ of genus $0$, and where for all $g \ge 0$, the $g$th level of $\mathcal{T}$ consists of all $n_g$ numerical semigroups of genus $g$. The construction of $\mathcal{T}$ is as follows \cite{RGGJ1, Br08}. Given a numerical semigroup $S$ of genus $g \ge 1$, its \emph{parent} is $\hat{S}= S \cup \{f\}$ where $f$ is the Frobenius number of $S$. Then $\hat{S}$ is also a numerical semigroup, of genus $g-1$. Here are the first five levels of $\mathcal{T}$. 

\hspace{-0.7cm}
\begin{tikzpicture}[scale=1.45]

\coordinate (A) at (4.5,4);

\coordinate (B) at (4.5,3);

\coordinate (C1) at (3,2);
\coordinate (C2) at (6,2);

\coordinate (D1) at (1.5,1);
\coordinate (D2) at (4,1);
\coordinate (D3) at (5.5,1);
\coordinate (D4) at (7,1);

\coordinate (E1) at (-0.4,0);
\coordinate (E2) at (1,0);
\coordinate (E3) at (2.1,0);
\coordinate (E4) at (3.05,0);
\coordinate (E5) at (4,0);
\coordinate (E6) at (5,0);
\coordinate (E7) at (7.5,0);

\draw [fill] (A) circle (0.06) node [right] {$\vs{1}$};

\draw [fill] (B) circle (0.06) node [right] {$\vs{2,3}$};

	\draw [thick] (A)--(B);

\draw [fill] (C1) circle (0.06) node [left] {$\vs{3,4,5}$\,};
\draw [fill] (C2) circle (0.06) node [right] {$\vs{2,5}$};

	\draw [thick] (B)--(C1);
	\draw [thick] (B)--(C2);

\draw [fill] (D1) circle (0.06) node [left] {$\vs{4,5,6,7}$\,};
\draw [fill] (D2) circle (0.06) node [left] {$\vs{3,5,7}$};
\draw [fill] (D3) circle (0.06) node [right] {$\vs{3,4}$};
\draw [fill] (D4) circle (0.06) node [right] {$\vs{2,7}$};

	\draw [thick] (C1)--(D1);
	\draw [thick] (C1)--(D2);
	\draw [thick] (C1)--(D3);
	\draw [thick] (C2)--(D4);
	
\draw [fill] (E1) circle (0.06) node [below] {$\vs{5,6,7,8,9}$};
\draw [fill] (E2) circle (0.06) node [below] {$\vs{4,6,7,9}$};
\draw [fill] (E3) circle (0.06) node [below] {$\vs{4,5,7}$};
\draw [fill] (E4) circle (0.06) node [below] {$\vs{4,5,6}$};
\draw [fill] (E5) circle (0.06) node [below] {$\vs{3,7,8}$};
\draw [fill] (E6) circle (0.06) node [below] {$\vs{3,5}$};
\draw [fill] (E7) circle (0.06) node [below] {$\vs{2,9}$};

	\draw [thick] (D1)--(E1);
	\draw [thick] (D1)--(E2);
	\draw [thick] (D1)--(E3);
	\draw [thick] (D1)--(E4);
	\draw [thick] (D2)--(E5);
	\draw [thick] (D2)--(E6);
	\draw [thick] (D4)--(E7);	

\end{tikzpicture}

This illustrates the data $(n_0,n_1,n_2,n_3,n_4) = (1,1,2,4,7)$ given earlier.

\medskip
As an illustration of Zhai's result that `most' numerical semigroups are of depth $q \le 3$, Figure~\ref{12 levels} displays the first $12$ levels of $\mathcal{T}$, where the numerical semigroups of depth $q \le 3$ and $q \ge 4$ are represented by black dots and smaller gray dots, respectively. The bottom line consists of $n_{11}=343$ dots, among which there are $287$ black ones.

\begin{center}
\begin{figure}
\input{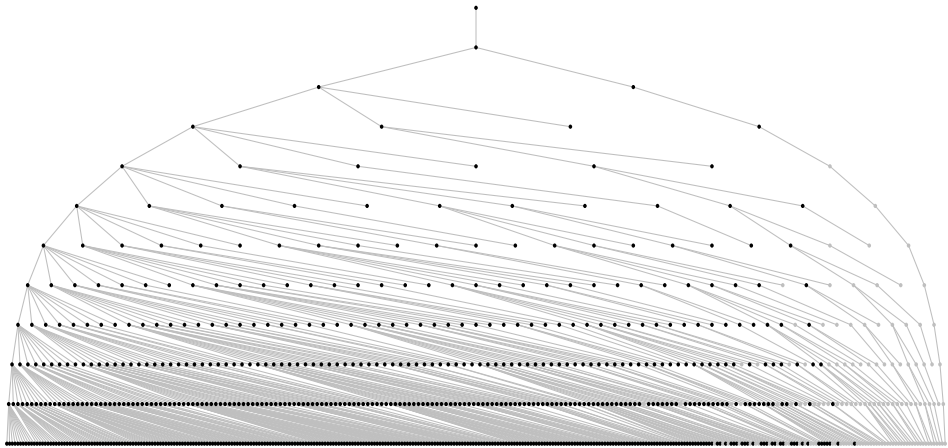}
\caption{The first $12$ levels of $\mathcal{T}$. Black dots correspond to depth $q \le 3$.}
\label{F:Cloche}
\label{12 levels}
\end{figure}
\end{center}

\section{Gapsets}\label{section gapsets}
\begin{definition} A \emph{gapset} is a finite set $G \subset \N_+$ satisfying the following property: for all $z \in G$, if $z=x+y$ with $x,y \in \N_+$, then $x \in G$ or $y \in G$.
\end{definition}
Notice the similarity of this definition with that of a prime ideal $P$ in a ring $R$, where for any $z\in P$, any decomposition $z=xy$ with $x,y \in R$ implies $x \in P$ or $y \in P$. 

\begin{remark} It follows from the definition that a gapset $G$ is nothing else than the set of gaps of a numerical semigroup $S$, where $S = \N \setminus G$.
\end{remark}

One of our purposes here is to show that thinking in terms of gapsets rather than numerical semigroups proper may lead to advances on the latter. In particular, this is what originally led us to the present partial results on the Bras-Amor\'os conjecture $n_{g} \ge n_{g-1}+n_{g-2}$. Indeed, as will become clear in this paper, gapsets may be manipulated and transformed in ways which are not so conveniently expressible on the level of numerical semigroups.

\smallskip
We now transfer in a natural way some terminology from numerical semigroups to gapsets.

\begin{definition} Let $G \subset \N_+$ be a gapset. The \emph{multiplicity} of $G$ is the smallest integer $m \ge 1$ such that $m \notin G$, the \emph{Frobenius number} of $G$ is  $f = \max G$ if $G \not=\emptyset$ and $f=-1$  otherwise, the \emph{conductor} of $G$ is $c=f+1$, the \emph{genus} of $G$ is $g=\card(G)$ and the \emph{depth} of $G$ is $q=\lceil c/m \rceil$.
\end{definition}

That is, the multiplicity, Frobenius number, conductor, genus and depth of a nonempty gapset $G$ coincide with the corresponding numbers of its associated numerical semigroup $S = \N \setminus G$. 

\subsection{Revisiting $\mathcal{T}$ in terms of gapsets}

In order to reconstruct the tree $\mathcal{T}$ of numerical semigroups in the setting of gapsets, we need the following lemma.

\begin{lemma}\label{initial} Every initial segment of a gapset is a gapset. 
\end{lemma}
\begin{proof} Let $G$ be a gapset. Let $t \in \N_+$ and $G'=G \cap [1,t]$. We claim that $G'$ is a gapset. Let $z \in G'$, and assume $z=x+y$ with $1 \le x \le y$. Since $z \in G$, we have $\{x,y\} \cap G \not= \emptyset$, whence $\{x,y\} \cap G' \not= \emptyset$ since $x,y \le z \le t$. Therefore $G'$ is a gapset, as claimed.
\end{proof}

In particular, if $G$ is a nonempty gapset, then $G \setminus \{\max G\}$ is still a gapset. Plainly, designating the latter as the \emph{parent} of the former exactly captures the parenthood in $\mathcal{T}$. This gives the following figure which is much easier to understand than the above classical one.

\bigskip

\hspace{-0.7cm}
\begin{tikzpicture}[scale=1.5]

\coordinate (A) at (4.5,4);

\coordinate (B) at (4.5,3);

\coordinate (C1) at (3,2);
\coordinate (C2) at (6,2);

\coordinate (D1) at (1.5,1);
\coordinate (D2) at (4,1);
\coordinate (D3) at (5.5,1);
\coordinate (D4) at (7,1);

\coordinate (E1) at (-0.4,0);
\coordinate (E2) at (1,0);
\coordinate (E3) at (2,0);
\coordinate (E4) at (3,0);
\coordinate (E5) at (4,0);
\coordinate (E6) at (5,0);
\coordinate (E7) at (7.5,0);

\draw [fill] (A) circle (0.06) node [right] {$\emptyset$};

\draw [fill] (B) circle (0.06) node [right] {\,$1$};

	\draw [thick] (A)--(B);

\draw [fill] (C1) circle (0.06) node [left] {$12$\,\,};
\draw [fill] (C2) circle (0.06) node [right] {\,$13$};

	\draw [thick] (B)--(C1);
	\draw [thick] (B)--(C2);

\draw [fill] (D1) circle (0.06) node [left] {$123$\,\,};
\draw [fill] (D2) circle (0.06) node [left] {$124$};
\draw [fill] (D3) circle (0.06) node [right] {$125$};
\draw [fill] (D4) circle (0.06) node [right] {$135$};

	\draw [thick] (C1)--(D1);
	\draw [thick] (C1)--(D2);
	\draw [thick] (C1)--(D3);
	\draw [thick] (C2)--(D4);
	
\draw [fill] (E1) circle (0.06) node [below] {$1234$};
\draw [fill] (E2) circle (0.06) node [below] {$1235$};
\draw [fill] (E3) circle (0.06) node [below] {$1236$};
\draw [fill] (E4) circle (0.06) node [below] {$1237$};
\draw [fill] (E5) circle (0.06) node [below] {$1245$};
\draw [fill] (E6) circle (0.06) node [below] {$1247$};
\draw [fill] (E7) circle (0.06) node [below] {$1357$};

	\draw [thick] (D1)--(E1);
	\draw [thick] (D1)--(E2);
	\draw [thick] (D1)--(E3);
	\draw [thick] (D1)--(E4);
	\draw [thick] (D2)--(E5);
	\draw [thick] (D2)--(E6);
	\draw [thick] (D4)--(E7);	

\end{tikzpicture}

\medskip
Conversely, the \emph{children} of a gapset $G$ in $\mathcal{T}$  are exactly those gapsets $H$ such that $H = G \sqcup \{a\}$ for some $a > \max G$. This is a finite set, since if $G$ is of multiplicity $m$ and conductor $c$, then any child $H = G \sqcup \{a\}$ of $G$ satisfies $c \le a \le m+c-1$; for if $a \ge c+m$, then $G \sqcup \{a\}$ cannot be a gapset as it contains $a=m+(a-m)$, but neither $m$ nor $a-m$ since $a-m \ge c > \max G$.

\subsection{The canonical partition}

\begin{lemma} Let $G$ be a gapset of multiplicity $m$. Then 
\begin{eqnarray*}
[1,m-1] & \subseteq & G, \\
G \cap m\N & = & \emptyset.
\end{eqnarray*}
\end{lemma}
\begin{proof}  By definition of the multiplicity, $G$ contains $[1,m-1]$ but not $m$. Let $a \ge 2$ be an integer. The formula $am=m+(a-1)m$ and induction on $a$ imply that $am \notin G$.
\end{proof}
 This motivates the following notation and definition.

\begin{notation} Let $G$ be a gapset of multiplicity $m$. We denote $G_0=[1,m-1]$ and, more generally,
\begin{equation}\label{G_i}
G_i = G \cap [im+1,(i+1)m-1] \quad \mbox{ for all } i \ge 0.
\end{equation}
\end{notation}

\begin{proposition}\label{subset} Let $G$ be a gapset of multiplicity $m$ and depth $q$. Let $G_i$ be defined as in \eqref{G_i}. Then
\begin{equation}\label{partition}
G = G_0 \sqcup G_1 \sqcup \dots \sqcup G_{q-1}
\end{equation}
and $G_{q-1} \not= \emptyset$. Moreover $G_{i+1} \subseteq m+G_{i}$ for all $i \ge 0$. 
\end{proposition}
\begin{proof} As $G \cap m\N = \emptyset$, it follows that $G$ is the disjoint union of the $G_i$ for $i \ge 0$. Let $c$ be the conductor of $G$. Then $G \subseteq [1,c-1]$. Since $(q-1)m < c \le qm$ by definition of $q$, it follows that $G_i = \emptyset$ for $i \ge q$, whence \eqref{partition}. Let $f=c-1$. Since $f \in G$ and $f \ge (q-1)m+1$, it follows that $f \in G_{q-1}$. 

It remains to show that $G_{i+1} \subseteq m+G_{i}$ for all $i \ge 0$.
Let $x \in G_{i+1}$. Since $G_{i+1} \subseteq [(i+1)m+1, (i+2)m-1]$, we have 
$$
x-m \in [im+1, (i+1)m-1].  
$$
Now $x-m \in G$ since $x = m + (x-m)$ and $m \notin G$. So $x-m \in G_i$.
\end{proof}
\begin{definition} Let $G$ be a gapset. The \emph{canonical partition} of $G$ is the partition $G = G_0 \sqcup G_1 \sqcup \dots \sqcup G_{q-1}$ given by \eqref{partition}.
\end{definition}

\begin{remark}\label{data gapsets}
The multiplicity $m$, genus $g$ and depth $q$ of a gapset $G$ may be read off from its canonical partition $G = \sqcup_i G_i$ as follows : 
\begin{eqnarray*}
m & = & \max(G_0) +1, \\
g & = & \sum_i |G_i|, \\
q & = & \textrm{the number of nonempty } G_i.
\end{eqnarray*}
\end{remark}

\subsection{On $m$-extensions and $m$-filtrations}
We shall need to consider somewhat more general finite subsets of $\N_+$ than gapsets proper.

\begin{definition} Let $m \in \N_+$. An \emph{$m$-extension} is a finite set $A \subset \N_+$ containing $[1,m-1]$ and admitting a partition
\begin{equation}\label{m-ext}
A = A_0 \sqcup A_1 \sqcup \dots \sqcup A_{t}
\end{equation}
for some $t \ge 0$, where $A_0=[1,m-1]$ and $A_{i+1} \subseteq m+A_i$ for all $i \ge 0$. 
\end{definition}
In particular, an $m$-extension $A$ satisfies $A \cap m\N = \emptyset$. Moreover, the  above conditions on the $A_i$ imply 
\begin{equation}\label{A_i}
A_i=A \cap [im+1, (i+1)m-1]
\end{equation}
for all $i \ge 0$, whence the $A_i$ are \emph{uniquely determined} by $A$.

\begin{remark}\label{gapset is extension}
Every gapset of multiplicity $m$ is an $m$-extension. This follows from Proposition~\ref{subset}.
\end{remark}

\smallskip
Closely linked is the notion of \emph{$m$-filtration}.
\begin{definition}\label{filtration} Let $m \in \N_+$. An \emph{$m$-filtration} is a finite sequence $F=(F_0,F_1,\dots,F_t)$ of nonincreasing subsets of $\N_+$ such that
$$
F_0=[1,m-1] \supseteq F_1 \supseteq \dots \supseteq F_t.
$$ 
\end{definition}
For $m \in \N_+$, there is a straightforward bijection between $m$-extensions and $m$-partitions.

\begin{proposition} Let $A=A_0 \sqcup A_1 \sqcup \dots \sqcup A_{t}$ be an $m$-extension. Set $F_i=-im+A_i$ for all $i$. Then
$(F_0,F_1,\dots,F_t)$ is an $m$-filtration. Conversely, let $(F_0,F_1,\dots,F_t)$ be an $m$-filtration. Set $A_i = im+F_i$ for all $i$, and let $A$ be the union of the $A_i$. Then $A$ is an $m$-extension.
\end{proposition}
\begin{proof} We have $F_i=-im+A_i$ if and only if $A_i=im+F_i$.
\end{proof}

\begin{notation}\label{phi and tau} If $A$ is an $m$-extension, we denote by $F=\varphi(A)$ its associated $m$-filtration. Conversely, if $F$ is an $m$-filtration, we denote by $A=\tau(F)$ its associated $m$-extension. 
\end{notation}

By the above proposition, \emph{the maps $\varphi$ and $\tau$ are inverse to each other}.

\subsection{Gapset filtrations} 

\begin{definition}
Let $G \subset \N_+$ be a gapset of multiplicity $m$. The \emph{gapset filtration} associated to $G$ is the $m$-filtration $F=\varphi(G)$.
\end{definition}
By Remark~\ref{gapset is extension}, every gapset $G$ of multiplicity $m$ is an $m$-extension, whence $\varphi(G)$ is well-defined. 

\smallskip
Concretely, let $G$ be a gapset of multiplicity $m$ and depth $q$. As in \eqref{G_i}, let $G_i = G \cap [im+1,(i+1)m-1]$ for all $i \ge 0$, so that $G_0=[1,m-1]$ and
$$
G = G_0 \sqcup \dots \sqcup G_{q-1}.
$$
The associated $m$-filtration $F=\varphi(G)$ is then given by $F=(F_0,\dots,F_{q-1})$ where $F_i = -im+G_i$ for all $i \ge 0$.

\smallskip
We now transfer some terminology from gapsets to gapset filtrations.
\begin{definition} Let $F$ be a gapset filtration. Its \emph{multiplicity}, \emph{genus} and \emph{depth} are defined as those of its corresponding gapset $G=\tau(F)$.
\end{definition}

Let $F=(F_0,F_1,\dots,F_{q-1})$ be a gapset filtration, and let $G=\tau(F)=\cup_i G_i$ be the corresponding gapset. It follows from Remark~\ref{data gapsets} and the equality $|F_i|=|G_i|$ for all $i$, that the genus of $F$ is equal to $|F_0|+\dots+|F_{q-1}|$ and that its depth is equal to the number of nonzero $F_i$.

\subsection{An example: the case of genus $6$}
For illustration purposes, here are the 23 numerical semigroups or gapsets of genus 6 given in two different ways.

\medskip
{\tiny $\bullet$} With minimal generators:
{\small
\begin{align*}
& \vs{2,13}; \vs{3,7}; \vs{3,8,13}; \vs{3,10,11}; \vs{4,5}; \vs{4,6,9}; \vs{4,6,11,13}; \vs{4,7,9}; \\
& \vs{4,7,10,13}; \vs{4,9,10,11}; \vs{5,6,7}; \vs{5,6,8}; \vs{5,6,9,13}; \vs{5,7,8,9};  \vs{5,7,8,11}; \\
&  \vs{5,7,9,11,13}; \vs{5,8,9,11,12}; \vs{6,7,8,9,10};  \vs{6,7,8,9,11};  \vs{6,7,8,10,11}; \\
& \vs{6,7,9,10,11}; \vs{6,8,9,10,11,13}; \vs{7,8,9,10,11,12,13}.
\end{align*}
}

{\tiny $\bullet$} With the associated gapset filtration:
{\small
\begin{align*}
& (1)^6; (1 2)^2 (2)^2; (1 2)^2 (1)^2; (1 2)^3; (1 2 3) (2 3) (3); (1 2 3) (1 3) (3); (1 2 3) (1 3)(1); \\
& (1 2 3) (1 2) (2); (1 2 3) (1 2) (1); (1 2 3)^2; (1 2 3 4) (3 4); (1 2 3 4) (2 4); (1 2 3 4) (2 3);\\
& (1 2 3 4) (1)^2; (1 2 3 4) (1 4); (1 2 3 4) (1 3); (1 2 3 4) (1 2); (1 2 3 4 5) (5); (1 2 3 4 5) (4); \\
& (1 2 3 4 5) (3); (1 2 3 4 5) (2);  (1 2 3 4 5) (1); (1 2 3 4 5 6).
\end{align*}
}

The above notation is self-explanatory. For instance, the third element $(1 2)^2 (1)^2$ represents the filtration $F=(\{1,2\}, \{1,2\},\{1\},\{1\})$ of multiplicity $3$, corresponding to the gapset $G=\{1,2\} \cup \{4,5\} \cup \{7\} \cup \{10\}$ and the numerical semigroup $S = \N \setminus G = \{0,3,6,8, 9,11,12,13,\dots\} = \vs{3,8,13}$.

\begin{remark} Both descriptions of a numerical semigroup $S$, namely with minimal generators and with the associated gapset filtration, reveal its multiplicity $m$. The first one reveals its embedding dimension $e$, while the second one reveals its Frobenius number $f$, its genus $g$ and its depth $q$.
\end{remark}

\subsection{Some notation}

We shall denote by $\Gamma$ the set of all gapsets, and by $\Gamma(g)$ the subset of all gapsets of genus $g$. Similarly, we shall denote by $\fc$ the set of all gapset filtrations, and by $\fc(g)$ the subset of all gapset filtrations of genus $g$. Of course, the above maps  $\varphi$ and $\tau$ provide bijections between $\Gamma(g)$ and $\fc(g)$ for all $g \ge 0$. Thus, we have
$$
n_g \, =\,  |\Gamma(g)| \, =\, |\fc(g)|
$$
for all $g \ge 0$.

Furthermore, given $b \in \N_+$, we shall denote by $\Gamma(q \le b)$ the subset of all gapsets of depth $q \le b$, and by $\fc(q \le b)$ the corresponding subset for gapset filtrations. For a fixed genus, we shall denote by $\Gamma(g, q \le b)$ and $\fc(g,q \le b)$ the subsets of $\Gamma(q \le b)$ and $\fc(q \le b)$ of elements of genus $g$, respectively.

The case $b=3$ is of special importance here. Thus, as in the Introduction, we set
$$
n_g' \, =\,  |\Gamma(g, q \le 3)| \, =\, |\fc(g, q \le 3)|
$$
for all $g \ge 0$.

\section{The case $q \le 2$}\label{section q <= 2}

Zhao established that \emph{the number of numerical semigroups of genus $g \ge 0$ and depth $q \le 2$ is equal to the Fibonacci number $\F_{g+1}$}~\cite{Zhao}. His proof, based on counting arguments, uses formulas expressing the Fibonacci numbers as sums of binomial coefficients. Here we prove Zhao's result bijectively, without counting.

\begin{lemma}\label{filtrations of depth 2} Let $m \ge 1$. Every $m$-filtration $(F_0,F_1)$ is a gapset filtration of multiplicity $m$ and depth $q \le 2$, and conversely.
\end{lemma}
\begin{proof} Let $F=(F_0,F_1)$ be an $m$-filtration. Thus $F_0 =[1,m-1]$ and $F_1 \subseteq F_0$. Let $G=\tau(F) = G_0 \cup G_1$, i.e. $G_0=F_0$ and $G_1 = m+F_1$. Let $z \in G$. Assume $z=x+y$ with $x \le y$ positive integers. If $z \in G_0$ then $x,y \in G_0$. If  $z \in G_1$, then $z \le 2m-1$, whence $x \le m-1$ and so $x \in G_0$. Therefore $G$ is a gapset, of depth $q \le 2$, and $F=\varphi(G)$ is a gapset filtration. The converse holds by definition.
\end{proof}

\begin{proposition} For all $g \ge 2$, we have
\begin{equation}\label{q <= 2}
|\mathcal{F}(g, q \le 2)| \,=\, |\mathcal{F}(g-1, q \le 2)|+|\mathcal{F}(g-2, q \le 2)|.
\end{equation}
\end{proposition}
\begin{proof} Let $F \in \mathcal{F}(g, q \le 2)$. Let $m \ge 1$ be its multiplicity. Then $F=(F_0,F_1)$ where $F_0=[1,m-1]$ and $F_1 \subseteq F_0$. We have $g=|F_0|+|F_1|$. 

{\tiny $\bullet$} If $g=0$ then $F_0=F_1=\emptyset$, corresponding to the case $m=1$. Thus $|\mathcal{F}(0, q \le 2)| = 1$. 

{\tiny $\bullet$} If $g=1$ then $F_0=\{1\}$ and $F_1=\emptyset$, so that $|\mathcal{F}(1, q \le 2)| = 1$. 

{\tiny $\bullet$} If $g=2$, then either $F_0=[1,2], F_1=\emptyset$ or else $F_0=F_1=\{1\}$.

Hence $|\mathcal{F}(2, q \le 2)| = 2$ and \eqref{q <= 2} holds for $g =2$.

{\tiny $\bullet$} Assume now $g \ge 3$. As $|F_0|+|F_1|=g$, we have $|F_0| \ge 2$, whence $m \ge 3$ since $F_0=[1,m-1]$.

{\tiny $\bullet$ \hspace{-1.5mm} $\bullet$} If $\max(F_1) \le m-2$, let $F_0'=F_0 \setminus \{m-1\}=[1,m-2]$, $F_1'=F_1$ and $F'=(F_0',F_1')$. Then $F'$ is a gapset filtration by Lemma~\ref{filtrations of depth 2}, of genus $g-1$. That is, $F' \in \mathcal{F}(g-1, q \le 2)$.

{\tiny $\bullet$ \hspace{-1.5mm} $\bullet$} If $\max(F_1) = m-1$, let $F_i''=F_i \setminus \{m-1\}$ for $i=0,1$, and let $F''=(F_0'', F_1'')$. Then $F'' \in \mathcal{F}(g-2, q \le 2)$ by Lemma~\ref{filtrations of depth 2}.

\smallskip
Clearly the maps $F \mapsto F'$ and $F \mapsto F''$, where applicable, are one-to-one, and their respective domains cover the whole of $\mathcal{F}(g, q \le 2)$.  It follows that 
$$
|\mathcal{F}(g, q \le 2)| \,\le\, |\mathcal{F}(g-1, q \le 2)|+|\mathcal{F}(g-2, q \le 2)|.
$$

\smallskip
Conversely, let $F=(F_0,F_1)$ be a gapset filtration of depth $q \le 2$. Let $m$ be its multiplicity, so that $F_0=[1,m-1]$. Let 
$$
\widehat{F_0}=F_0\cup\{m\}, \,\, \widehat{F_1}=F_1\cup\{m\}, \,\, \widehat{F}=(\widehat{F_0},F_1), \,\, \wh{F}=(\widehat{F_0},\widehat{F_1}).
$$
Then both $\widehat{F}$ and $\wh{F}$ are gapset filtrations by Lemma~\ref{filtrations of depth 2}. Moreover, we have $\genus(\widehat{F})=\genus(F)+1$ and $\genus(\wh{F})=\genus(F)+2$. Finally, the maps $F \mapsto \widehat{F}$ and $F \mapsto \wh{F}$ are one-to-one and have disjoint images in $\mathcal{F}(q \le 2)$, since gapset filtrations of the form $\wh{F}$ are \emph{characterized} by the property that their two pieces have the same maximal element $m$. Therefore
$$
|\mathcal{F}(g, q \le 2)| \,\ge\, |\mathcal{F}(g-1, q \le 2)|+|\mathcal{F}(g-2, q \le 2)|
$$
and the proof is complete.
\end{proof}

\begin{corollary}\label{fibo} For all $g \ge 0$, we have $|\mathcal{F}(g, q \le 2)|=\F_{g+1}$, where $\F_n$ denote the $n$th Fibonacci number. 
\end{corollary}
\begin{proof} The formula holds for $g=0,1$. By \eqref{q <= 2}, the numbers $|\mathcal{F}(g, q \le 2)|$ satisfy the same recurrence relation as the Fibonacci numbers. Hence the formula holds for all $g \ge 0$.
\end{proof}

\section{A lower bound on $n_g'$}\label{section lower}

Recall that we denote by $\fc$ the set of all gapsets, and by $\fc(g)$ the subset of all those of genus $g$. Moreover, given a set $\mathcal{C}$ of conditions, we denote by $\fc(\mathcal{C})$ and $\fc(g,\mathcal{C})$ the subset of elements of $\fc$ and $\fc(g)$ satisfying $\mathcal{C}$, respectively. 

Similar constructions as for $q \le 2$ will work for $q \le 3$. Thus, we shall define two self-maps on $\fc(q \le 3)$ which increase the multiplicity by $1$, and the genus by 1 and 2, respectively. 

\begin{notation} Let $m \in \N_+$, and let $F=(F_0,F_1,F_2)$ be an $m$-filtration of length at most $3$, so that $F_0=[1,m-1] \supseteq F_1 \supseteq F_2$. We denote 
\begin{eqnarray*}
\alpha_1(F) & = & (F_0 \sqcup \{m\}, F_1, F_2), \\
\alpha_2(F) & = & (F_0 \sqcup \{m\}, F_1 \sqcup \{m\}, F_2).
\end{eqnarray*}
\end{notation}
Then both $\alpha_1(F), \alpha_2(F)$ are $(m+1)$-filtrations, as their first part is $[1,m]$.

\begin{proposition}\label{alpha} Let $F=(F_0,F_1,F_2)$ be a gapset filtration of genus $g$. Then $\alpha_1(F), \alpha_2(F)$ are gapset filtrations of genus $g+1, g+2$, respectively.
\end{proposition}

\begin{proof} Let $m$ be the multiplicity of $F$, so that $F_0=[1,m-1]$. Let 
$$G = \tau(F)=G_0 \sqcup G_1 \sqcup G_2$$ 
be the corresponding gapset, i.e. with $G_1=m+F_1$ and $G_2=2m+F_2$. We have $G_1 \subseteq [m+1, 2m-1]$ and $G_2 \subseteq [2m+1, 3m-1]$. 

Let $H=H_0 \sqcup H_1 \sqcup H_2$ be the $(m+1)$-extension corresponding to $\alpha_1(F)$, i.e. $H = \tau(\alpha_1(F))$. Then
$$
\begin{array}{rcl}
H_0 & = & F_0 \sqcup \{m\}=[1,m], \\
H_1 & = & (m+1)+F_1, \\
H_2 & = & 2(m+1)+F_2,
\end{array}
$$
so that $H_1=1+G_1$, $H_2=2+G_2$. It follows that
\begin{eqnarray}
H_1 & \subseteq & [m+2,2m], \label{inclusion 1} \\
H_2 & \subseteq & [2m+3,3m+1]. \label{inclusion 2}
\end{eqnarray}
Note that $|G|=g$ and $|H|=g+1$. We claim that $H$ is a gapset. Let $z \in H$, and assume $z = x+y$ with $x,y$ integers such that $1 \le x \le y$. We need show
\begin{equation}\label{H gapset}
x \in H \textrm{\, or\, } y \in H.
\end{equation}

\noindent
{\tiny $\bullet$}  If $z \in H_0$, then $x,y \in H_0$ as well and we are done.

\noindent
{\tiny $\bullet$}  Assume $z \in H_1$. Then $z \le 2m$ by \eqref{inclusion 1}. Hence $x \le m$, \emph{i.e.} $x \in H_0$ and we are done.

\noindent
{\tiny $\bullet$}  Finally, assume $z \in H_2$. Then $z \le 3m+1$  by \eqref{inclusion 2}. If $x \le m$, then $x \in H_0$ and we are done. Assume now $x \ge m+1$. Then $y \le 2m$ since $z \le 3m+1$.  Consider 
$$z'=z-2=(x-1)+(y-1).$$ 
Then $z' \in G_2$ by construction. Since $G$ is a gapset, it follows that $x-1 \in G$ or $y-1 \in G$. More precisely, since $m \le x-1 \le y-1 \le 2m-1$, we have $x-1 \in G_1$ or $y-1 \in G-1$. Hence $x \in H_1$ or $y \in H_1$ and so \eqref{H gapset} again holds, as desired. 

\smallskip
Let now $H'$ be the $(m+1)$-extension corresponding to the filtration $\alpha_2(F)$, i.e. $H'=\tau(\alpha_2(F))$. Then $H'=H \cup \{2m+1\}$. Since $2m+1 \notin H$, we have $|H'|=|H|+1=g+2$. We have already shown that $H$ is a gapset. In order to show that $H'$ also is, it remains to show that for any integer decomposition $2m+1=x+y$ with $1 \le x \le y$, we have $x \in H'$ or $y \in H'$. But this is easy, since then $x \le m$ and so $x \in H'$. 

\smallskip
We conclude, as claimed, that $\alpha_1(F), \alpha_2(F)$ are gapset filtrations of genus $g+1, g+2$, respectively. Both are of depth $q \le 3$ and multiplicity $m+1$, since they contain $[1,m]$ but not $m+1$.
\end{proof}

Note that the corresponding statement is no longer true in general for depth $q \ge 4$. For instance, $(1)^4$ is a gapset filtration of multiplicity $2$ and depth $4$, but $(12)(1)^3$ is no longer a gapset filtration, since its associated set $G=\tau((12)(1)^3)=\{1,2\} \sqcup \{4\} \sqcup \{7\}\sqcup \{10\}$ contains $10=5+5$ but not $5$ and therefore is not a gapset.

\smallskip
The above result implies that $\alpha_1, \alpha_2$ induce two well-defined injective maps
$$
\alpha_1, \alpha_2 \colon \mathcal{F}(q \le 3)  \to \mathcal{F}(q \le 3).
$$

\begin{proposition} We have $\im(\alpha_1) \cap \im(\alpha_2) = \emptyset$.
\end{proposition}
\begin{proof}
Let $F=(F_0,F_1,F_2) \in \mathcal{F}(q \le 3)$. It follows from the definition of the maps $\alpha_i$ that if $F \in \im(\alpha_1)$, then $\max F_0 > \max F_1$, whereas if $F \in \im(\alpha_2)$, then $\max F_0 = \max F_1$. Therefore $F$ cannot belong to both.
\end{proof}

\begin{corollary}\label{cor n'} For all $g \ge 2$, we have $n_g' \,\ge \, n_{g-1}'+n_{g-2}'$.
\end{corollary}
\begin{proof} Indeed, the above results imply that, for $g \ge 2$, the set $\fc(g, q \le 3)$ contains disjoint copies of $\fc(g-1, q \le 3)$ and $\fc(g-2, q \le 3)$. Whence the inequality
$|\fc(g, q \le 3)| \,\ge \, |\fc(g-1, q \le 3)|+|\fc(g-2, q \le 3)|$.
\end{proof}

\smallskip
Corollary~\ref{cor n'} and its proof are illustrated in Figure~\ref{567}. Let $\mathcal{T}'=\fc(q \le 3)$ considered as a subtree of $\mathcal{T}$. Then $\mathcal{T}'$ has $n'_g$ vertices at level $g$ for all $g \ge 0$. Figure~\ref{567} shows the levels $g=5,6,7$ of $\mathcal{T}'$. There are $n'_5=11$ vertices pictured 
\hspace{0.1pt}
\begin{tikzpicture}[x=1pt,y=1pt,scale=0.85,baseline=-3pt]
\fill [color=black] (0,0) circle(1.5);
\end{tikzpicture}
\hspace{0.1pt}
at level $5$, and $n'_6=20$ vertices pictured
\hspace{0.1pt}
\begin{tikzpicture}[x=1pt,y=1pt,scale=0.85,baseline=-3pt]
\fill [color=black!40!white] (0,0) circle(3);
\end{tikzpicture}
\hspace{0.1pt}
at level $6$. Level $7$ of $\mathcal{T}'$ is seen to contain disjoint copies of levels $5$ and $6$, plus two more vertices pictured
\hspace{0.1pt}
 \begin{tikzpicture}[x=1pt,y=1pt,scale=0.85,baseline=-3pt]
\fill [color=black!20!white] (0,0) circle(5);
\end{tikzpicture}, 
thereby accounting for all $n'_7=33$ vertices at that level.

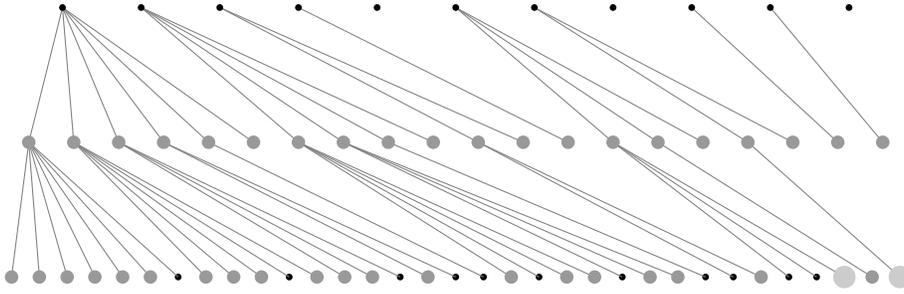
\begin{figure}
\begin{center}
\begin{tikzpicture}[x=1pt,y=1pt, scale=0.85]
\fill [color=white] (86.47058823529412,0) circle(5);
\draw [color=gray] (35.0,120) -- (20.0,60);
\draw [color=gray] (35.0,120) -- (40.0,60);
\draw [color=gray] (35.0,120) -- (60.0,60);
\draw [color=gray] (35.0,120) -- (80.0,60);
\draw [color=gray] (35.0,120) -- (100.0,60);
\draw [color=gray] (35.0,120) -- (120.0,60);
\fill [color=black] (35.0,120) circle(1.5);
\fill [color=black] (86.47058823529412,0) circle(1.5);
\fill [color=white] (135.88235294117646,0) circle(5);
\draw [color=gray] (70.0,120) -- (140.0,60);
\draw [color=gray] (70.0,120) -- (160.0,60);
\draw [color=gray] (70.0,120) -- (180.0,60);
\draw [color=gray] (70.0,120) -- (200.0,60);
\fill [color=black] (70.0,120) circle(1.5);
\fill [color=black] (135.88235294117646,0) circle(1.5);
\fill [color=white] (185.29411764705884,0) circle(5);
\draw [color=gray] (105.0,120) -- (220.0,60);
\draw [color=gray] (105.0,120) -- (240.0,60);
\fill [color=black] (105.0,120) circle(1.5);
\fill [color=black] (185.29411764705884,0) circle(1.5);
\fill [color=white] (210.0,0) circle(5);
\draw [color=gray] (140.0,120) -- (260.0,60);
\fill [color=black] (140.0,120) circle(1.5);
\fill [color=black] (210.0,0) circle(1.5);
\fill [color=white] (222.3529411764706,0) circle(5);
\fill [color=black] (175.0,120) circle(1.5);
\fill [color=black] (222.3529411764706,0) circle(1.5);
\fill [color=white] (247.05882352941177,0) circle(5);
\draw [color=gray] (210.0,120) -- (280.0,60);
\draw [color=gray] (210.0,120) -- (300.0,60);
\draw [color=gray] (210.0,120) -- (320.0,60);
\fill [color=black] (210.0,120) circle(1.5);
\fill [color=black] (247.05882352941177,0) circle(1.5);
\fill [color=white] (284.11764705882354,0) circle(5);
\draw [color=gray] (245.0,120) -- (340.0,60);
\draw [color=gray] (245.0,120) -- (360.0,60);
\fill [color=black] (245.0,120) circle(1.5);
\fill [color=black] (284.11764705882354,0) circle(1.5);
\fill [color=white] (321.1764705882353,0) circle(5);
\fill [color=black] (280.0,120) circle(1.5);
\fill [color=black] (321.1764705882353,0) circle(1.5);
\fill [color=white] (333.5294117647059,0) circle(5);
\draw [color=gray] (315.0,120) -- (380.0,60);
\fill [color=black] (315.0,120) circle(1.5);
\fill [color=black] (333.5294117647059,0) circle(1.5);
\fill [color=white] (358.2352941176471,0) circle(5);
\draw [color=gray] (350.0,120) -- (400.0,60);
\fill [color=black] (350.0,120) circle(1.5);
\fill [color=black] (358.2352941176471,0) circle(1.5);
\fill [color=white] (370.5882352941177,0) circle(5);
\fill [color=black] (385.0,120) circle(1.5);
\fill [color=black] (370.5882352941177,0) circle(1.5);
\draw [color=gray] (20.0,60) -- (12.352941176470589,0);
\draw [color=gray] (20.0,60) -- (24.705882352941178,0);
\draw [color=gray] (20.0,60) -- (37.05882352941177,0);
\draw [color=gray] (20.0,60) -- (49.411764705882355,0);
\draw [color=gray] (20.0,60) -- (61.76470588235294,0);
\draw [color=gray] (20.0,60) -- (74.11764705882354,0);
\draw [color=gray] (20.0,60) -- (86.47058823529412,0);
\fill [color=black!40!white] (20.0,60) circle(3);
\fill [color=black!40!white] (12.352941176470589,0) circle(3);
\draw [color=gray] (40.0,60) -- (98.82352941176471,0);
\draw [color=gray] (40.0,60) -- (111.1764705882353,0);
\draw [color=gray] (40.0,60) -- (123.52941176470588,0);
\draw [color=gray] (40.0,60) -- (135.88235294117646,0);
\draw [color=gray] (40.0,60) -- (148.23529411764707,0);
\fill [color=black!40!white] (40.0,60) circle(3);
\fill [color=black!40!white] (24.705882352941178,0) circle(3);
\draw [color=gray] (60.0,60) -- (160.58823529411765,0);
\draw [color=gray] (60.0,60) -- (172.94117647058823,0);
\draw [color=gray] (60.0,60) -- (185.29411764705884,0);
\fill [color=black!40!white] (60.0,60) circle(3);
\fill [color=black!40!white] (37.05882352941177,0) circle(3);
\draw [color=gray] (80.0,60) -- (197.64705882352942,0);
\draw [color=gray] (80.0,60) -- (210.0,0);
\fill [color=black!40!white] (80.0,60) circle(3);
\fill [color=black!40!white] (49.411764705882355,0) circle(3);
\draw [color=gray] (100.0,60) -- (222.3529411764706,0);
\fill [color=black!40!white] (100.0,60) circle(3);
\fill [color=black!40!white] (61.76470588235294,0) circle(3);
\fill [color=black!40!white] (120.0,60) circle(3);
\fill [color=black!40!white] (74.11764705882354,0) circle(3);
\draw [color=gray] (140.0,60) -- (234.7058823529412,0);
\draw [color=gray] (140.0,60) -- (247.05882352941177,0);
\draw [color=gray] (140.0,60) -- (259.4117647058824,0);
\draw [color=gray] (140.0,60) -- (271.7647058823529,0);
\fill [color=black!40!white] (140.0,60) circle(3);
\fill [color=black!40!white] (98.82352941176471,0) circle(3);
\draw [color=gray] (160.0,60) -- (284.11764705882354,0);
\draw [color=gray] (160.0,60) -- (296.47058823529414,0);
\draw [color=gray] (160.0,60) -- (308.8235294117647,0);
\fill [color=black!40!white] (160.0,60) circle(3);
\fill [color=black!40!white] (111.1764705882353,0) circle(3);
\draw [color=gray] (180.0,60) -- (321.1764705882353,0);
\fill [color=black!40!white] (180.0,60) circle(3);
\fill [color=black!40!white] (123.52941176470588,0) circle(3);
\fill [color=black!40!white] (200.0,60) circle(3);
\fill [color=black!40!white] (148.23529411764707,0) circle(3);
\draw [color=gray] (220.0,60) -- (333.5294117647059,0);
\draw [color=gray] (220.0,60) -- (345.88235294117646,0);
\fill [color=black!40!white] (220.0,60) circle(3);
\fill [color=black!40!white] (160.58823529411765,0) circle(3);
\fill [color=black!40!white] (240.0,60) circle(3);
\fill [color=black!40!white] (172.94117647058823,0) circle(3);
\fill [color=black!40!white] (260.0,60) circle(3);
\fill [color=black!40!white] (197.64705882352942,0) circle(3);
\draw [color=gray] (280.0,60) -- (358.2352941176471,0);
\draw [color=gray] (280.0,60) -- (370.5882352941177,0);
\draw [color=gray] (280.0,60) -- (382.94117647058823,0);
\fill [color=black!40!white] (280.0,60) circle(3);
\fill [color=black!40!white] (234.7058823529412,0) circle(3);
\draw [color=gray] (300.0,60) -- (395.29411764705884,0);
\fill [color=black!40!white] (300.0,60) circle(3);
\fill [color=black!40!white] (259.4117647058824,0) circle(3);
\fill [color=black!40!white] (320.0,60) circle(3);
\fill [color=black!40!white] (271.7647058823529,0) circle(3);
\draw [color=gray] (340.0,60) -- (407.64705882352945,0);
\fill [color=black!40!white] (340.0,60) circle(3);
\fill [color=black!40!white] (296.47058823529414,0) circle(3);
\fill [color=black!40!white] (360.0,60) circle(3);
\fill [color=black!40!white] (308.8235294117647,0) circle(3);
\fill [color=black!40!white] (380.0,60) circle(3);
\fill [color=black!40!white] (345.88235294117646,0) circle(3);
\fill [color=black!40!white] (400.0,60) circle(3);
\fill [color=black!40!white] (395.29411764705884,0) circle(3);
\fill [color=black!20!white] (382.94117647058823,0) circle(5);
\fill [color=black!20!white] (407.64705882352945,0) circle(5);
\end{tikzpicture}
\end{center}\caption{Disjoint embeddings of levels $5,6$ into level $7$ of $\mathcal{T}'$. }
\label{567}
\end{figure}
\begin{remark}
Replacing $q \le 3$ by $q = 4$ in Corollary~\ref{cor n'}, the corresponding inequality seems to hold for all $g \ge 2$, except for $g = 5$ since the values of $|\fc(g, q =4)|$ for $g=3,4,5$ are $0,1,0$. By contrast, for $q=5$, the corresponding inequality almost completely fails, at least apparently, as we conjecture that
$$
|\fc(g, q = 5)| \,< \, |\fc(g-1, q = 5)|+|\fc(g-2, q = 5)|
$$
holds for all $g \ge 21$.
\end{remark}
Nevertheless, replacing $q=d$ by $q \le d$ as in Corollary~\ref{cor n'} for $d=3$, we have the following conjecture.
\begin{conjecture} For all $d,g \ge 2$, we have
$$
|\fc(g, q \le d)| \,\ge \, |\fc(g-1, q \le d)|+|\fc(g-2, q \le d)|.
$$ 
\end{conjecture}
This conjecture may be seen as a refinement of the conjecture $n_g \ge n_{g-1}+ n_{g-2}$ for all $g \ge 2$. Corollaries~\ref{fibo} and ~\ref{cor n'} show that it holds for $d=2$ and $d=3$, respectively.

\section{An upper bound on $n_g'$}\label{section upper}

Having just proved the lower bound $n'_g \ge n_{g-1}'+n_{g-2}'$ for all $g \ge 2$, we now establish the upper bound
$$
n'_g \le n_{g-1}'+n_{g-2}'+n_{g-3}'
$$
for all $g \ge 3$.

\subsection{The images of $\alpha_1, \alpha_2$}

We first determine the respective images in $\fc(g, q \le 3)$ of the maps $\alpha_1, \alpha_2$.

\begin{proposition}\label{images} Let $F=(F_0,F_1,F_2) \in \fc(g, q \le 3)$ with $g \ge 2$. Then 
\begin{eqnarray*}
F \in \im(\alpha_1) & \iff & \max F_0 > \max F_1, \\
F \in \im(\alpha_2) & \iff & \max F_0 = \max F_1 > \max F_2.
\end{eqnarray*}
\end{proposition}
\begin{proof} Let $F=(F_0,F_1,F_2)$ be a gapset filtration of genus $g$. By construction of the map $\alpha_i$, the stated condition for $F$ to belong to $\im(\alpha_i)$ is necessary. We now prove that it is sufficient. Let $m \ge 2$ be the multiplicity of $F$, so that
$$
[1,m-1] = F_0 \supseteq F_1 \supseteq F_2.
$$
Thus $\max F_0 = m-1$, and we have two cases to consider: 

\smallskip
\textbf{Case 1.} $\max F_1 < m-1$.

\smallskip
\textbf{Case 2.} $\max F_2 < \max F_1 = m-1$.

\smallskip
\noindent
Thus $\max F_2 \le m-2$ in both cases. We claim that $F \in \im(\alpha_1)$ in Case 1 and $F \in \im(\alpha_2)$ in Case 2.

Let $G = \tau(F) = G_0 \cup G_1 \cup G_2$, where $G_i = im +F_i$ for $i=0,1,2$. Then $G$ is a gapset by hypothesis.

\smallskip
Consider the $(m-1)$-filtration $F' \, = \, (F'_0, F'_1,F'_2)$, where 
$$
F'_0 = F_0\setminus \{m-1\} = [1,m-2], \quad F'_1 = F_1\setminus \{m-1\}, \quad F'_2=F_2.
$$
Note that $F'_1 = F_1$ in Case 1. Let $G'$ be the corresponding $(m-1)$-extension, i.e. $G'=\tau(F')=G_0' \cup G_1' \cup G_2'$, where $G_i' = i(m-1) +F'_i$ for $i=0,1,2$. That is, we have 
$$G'_0=F'_0=[1,m-2], \quad G'_1=(m-1)+F'_1, \quad G'_2=2(m-1)+F_2.$$ 
Note that $|G'|=g-1$ in Case 1 and $|G'|=g-2$ in Case 2. We claim that $G'$ is a gapset in both cases.

\smallskip
To start with, it follows from Lemma~\ref{filtrations of depth 2} that $G_0' \cup G_1'$ is a gapset of depth at most 2. Let now $z \in G_2'$, and let $z = x+y$ with $1 \le x \le y$. We must show that $x$ or $y$ belongs to $G'$.

\smallskip
{\tiny $\bullet$} If $x \le m-2$ then $x \in F_0' = G_0' \subseteq G'$ and we are done in this case.

\smallskip
{\tiny $\bullet$} Assume now $x \ge m-1$. Since $z \in G_2'= 2(m-1)+F_2$, we have $z = 2(m-1)+t$ with $t \in F_2$. Since $\max F_2  \le m-2$ in both cases, we have $z \le 3m-4$. Since $x+y = z$ and $x \ge m-1$, it follows that $y \le 2m-3$. Thus so far, we have
$$
m-1 \le x \le y \le 2m-3.
$$
Now $(x+1)+(y+1)=z+2 = 2m+t$, and $2m+t \in G_2$ since $t \in F_2$. Therefore $x+1$ or $y+1$ belongs to $G$ since $G$ is a gapset. It follows that $x+1$ or $y+1$ belongs to $G_1$, as $m \le x+1 \le y+1 \le 2m-2$ and $\max G_0=m-1$, $\min G_2 \ge 2m+1$. Therefore $x$ or $y$ belongs to $G_1-1 = (m+F_1)-1=(m-1)+F_1$. 

\smallskip
{\tiny $\bullet$} \underline{In Case 1}, we have $F'_1=F_1$, thus $x$ or $y$ belongs to $(m-1)+F'_1=G'_1 \subseteq G'$ and we are done in this case. We conclude that $F'$ is a gapset filtration on genus $g-1$, and $F= \alpha_1(F')$ by construction, so that $F \in \im(\alpha_1)$.

\smallskip
{\tiny $\bullet$} \underline{In Case 2}, we have $F'_1=F_1 \setminus \{m-1\}$.
Now since $x \le y \le 2m-3=(m-1)+(m-2)$, it follows that $x$ or $y$ in fact belongs to $(m-1)+F_1\setminus \{m-1\}=(m-1)+F_1' = G_1'$. Whence $x$ or $y$ belongs to $G'$ as desired. We conclude here again that $F'$ is a gapset filtration, now of genus $g-2$, and $F=\alpha_2(F'')$ by construction, so that $F \in \im(\alpha_2)$.
\end{proof}

\subsection{Trimming the maximal elements}

\begin{proposition} Let $F=(F_0, F_1,F_2)$ be a gapset filtration of multiplicity $m+1 \ge 2$ and depth $3$, so that $F_0=[1,m] \supseteq F_1 \supseteq F_2 \not=\emptyset$. Let $a_i= \max F_i$ and $F'_i=F_i \setminus \{a_i\}$ for all $i$. Let
$$F'=(F'_0, F'_1, F'_2).$$
Then $F'$ is a gapset filtration.
\end{proposition}
\begin{proof} We have $m = a_0 \ge a_1 \ge a_2 \ge 1$. Let $G = \tau(F)$ and $G'=\tau(F')$. 
Then $G =G_0 \sqcup G_1 \sqcup G_2$ and $G'=G'_0 \sqcup G'_1 \sqcup G'_2$, where
\begin{eqnarray*}
G_i & = & i(m+1)+F_i, \\
G'_i & = & im+F'_i
\end{eqnarray*}
for $0 \le i \le 2$ by construction. By hypothesis $G$ is a gapset, and we claim that $G'$ also is. If $F'_2 = \emptyset$, the claim is true since then $F'=(F'_0,F'_1)$ is of length at most 2. 

Assume now $F'_2 \not= \emptyset$. Let $z \in G'_2$. Then $z= 2m+b$ for some $b \in F'_2$, with $b < a_2 \le m$ by construction. Assume $z = x + y$ for some integers $1 \le x \le y$. It suffices to show $\{x,y\} \cap G' \not=\emptyset$ in order to conclude the proof.

{\tiny $\bullet$} If $x \le m-1$, we are done since then $x \in G'_0=F'_0=[1,m-1]$.

{\tiny $\bullet$} Assume now $x \ge m$. Since $x+y=z= 2m+b \le 3m-1$, it follows that $y \le 2m-1$. We have $z+2=2(m+1)+b \in G_2$. Since $G$ is a gapset and $z+2=(x+1)+(y+1)$, it follows that $\{x+1, y+1\} \cap G \not=\emptyset$. More precisely, since $\max G_0 = m$, $\min G_2 \ge 2m+3$ and $m+1 \le x+1 \le y+1 \le 2m$, we have $\{x+1, y+1\} \cap G_1 \not=\emptyset$. Subtracting $m+1$ yields
\begin{equation}\label{cap}
\{x-m, y-m\} \cap F_1 \not=\emptyset.
\end{equation}
If $\{x-m, y-m\} \cap F'_1 \not=\emptyset$, we are done since then $\{x, y\} \cap G'_1 \not=\emptyset$. Assume now $\{x-m, y-m\} \cap F'_1 =\emptyset$. Then \eqref{cap} implies $\{x-m, y-m\} \cap F_1 = \{a_1\}$. Therefore either $x-m=a_1$ or $y-m=a_1$, whence $y-m \ge a_1$ since $y \ge x$. But then
$$
2m+b=z=x+y \ge x+a_1+m.
$$
Hence $m+b \ge x+a_1$. This implies $x \le m+(b-a_1) < m$ since $b < a_2 \le a_1$ as noted earlier, a contradiction with the current case $x \ge m$. Hence the hypothesis $\{x-m, y-m\} \cap F'_1 =\emptyset$ is absurd. Therefore $\{x,y\} \cap G'_1 \not=\emptyset$ and the proof is complete.
\end{proof}

The above result no longer holds in general for depth $q \ge 4$. The smallest gapset filtration for which suppressing the max of the pieces fails to preserve the gapset property is $(123)(13)^3$. This corresponds to the gapset $\{1,2,3\}\cup\{5,7\}\cup\{9,11\}\cup\{13,15\}$. Suppressing the max's yields the filtration $(12)(1)^3$. It corresponds to the set $\{1,2\}\cup\{4\}\cup\{7\}\cup\{10\}$ which is not a gapset since it contains $10=5+5$ but not $5$.

\begin{corollary}\label{cor max} Let $(F_0,F_1,F_2)$ be a gapset filtration such that $\max F_0 = \max F_1 = \max F_2=m$. Let $F'_i = F_i \setminus \{m\}$. Then $(F'_0, F'_1, F'_2)$  is a gapset filtration.
\end{corollary}
\begin{proof} Follows as a special case of the above proposition.
\end{proof}

\subsection{Main result}
We are now in a position to state and prove our main estimates of $n_g'$.

\begin{theorem}\label{theorem main} Let $n'_g$ denote the number of gapsets of genus $g$ and depth $q \le 3$. Then
$$
n'_{g-1}+n'_{g-2} \,\le\, n'_{g} \,\le\, n'_{g-1}+n'_{g-2}+n'_{g-3}
$$
for all $g \ge 3$.
\end{theorem}
\begin{proof} The lower bound has been established in Corollary~\ref{cor n'}. We now prove the upper bound. The statement holds for $g=3,4,5$. Let $g \ge 6$, and set $X=\fc(g, q \le 3)$. Consider the partition $X = X_1 \sqcup X_2 \sqcup X_3$, where for $F=(F_0,F_1,F_2) \in X$, we set 
\begin{itemize}
\item $F \in X_1$ $\iff$ $\max(F_0) > \max(F_1)$, 
\item $F \in X_2$ $\iff$ $\max(F_0) = \max(F_1) > \max(F_2)$,
\item $F \in X_3$ $\iff$ $\max(F_0) = \max(F_1) = \max(F_2)$.
\end{itemize}
It follows from Proposition~\ref{images} that $F \in X_1$ if and only if $F \in \im(\alpha_1)$, and $F \in X_2$ if and only if $F \in \im(\alpha_2)$. Thus $|X_1|=n_{g-1}'$, $|X_2|=n_{g-2}'$.
Moreover, it follows from Corollary~\ref{cor max} that $X_3$ may be embedded into $\fc(g-3, q \le 3)$, by removing the common max of $F_0,F_1,F_2$. Whence 
$$|X\setminus (X_1 \sqcup X_2)| \le |\fc(g-3, q \le 3)|=n_{g-3}'$$
and the proof is complete.
\end{proof}

Recall that the \emph{tribonacci sequence} is the integer sequence $(\T_n)_{n \ge 0}$ defined recursively by $\T_0=0$, $\T_1=1$, $\T_2=1$ and $\T_{n}=\T_{n-1}+\T_{n-2}+\T_{n-3}$ for all $n \ge 3$. See e.g. Wikipedia~\cite{wiki}. The first few terms of this sequence are
$$
0,1,1,2,4,7,13,24,44,81,149,274,504,927,1705,\dots
$$
In analogy with the Fibonacci sequence, there is an exact formula for $\T_n$ in terms of the three roots of the polynomial $x^3-x^2-x-1$. The growth rate of this sequence is given by $\T_n/T_{n-1} \to t \sim 1.839$ as $n \to \infty$, where 
$$t = \frac{1+\sqrt[3]{19+3\sqrt{33}}+\sqrt[3]{19-3\sqrt{33}}}3$$ 
is the only real root of $x^3-x^2-x-1$ and is called the \emph{tribonacci constant}. 

\begin{corollary} For all $g \ge 3$, we have
$$
2\F_g \,\le\, n'_g \,\le\, \T_{g+1}.
$$
\end{corollary}
Note that the claimed inequality $n'_g \ge 2\F_g$ is a strengthening of the inequality $n_g \ge 2\F_g$ proved in \cite{Br09}.
\begin{proof} We have $(n_2',n_4')=(2,4)=(2\F_2,2\F_3)$. Since $n_{g+2}' \ge n_{g+1}'+n_{g}'$ and $\F_{g+2} = \F_{g+1}+\F_{g}$ for  all $g \ge 0$, the inequality $n_{g}' \ge 2\F_{g}$ follows by induction on $g$.

As for the upper bound, we have $(n_1',n_2',n_3') = (1,2,4)=(\T_2,\T_3,\T_4)$. By Theorem~\ref{theorem main} and the recurrence relation of the $\T_n$, the inequality $n'_g \,\le\, \T_{g+1}$ again follows by induction on $g$.
\end{proof}

Going to higher genus yields better estimates.
\begin{corollary} For all $g \ge 58$, we have
$$
7.8\F_g \,\le\, n'_g \,\le\, \frac{7}{1000}\T_{g+1}.
$$
\end{corollary}
\begin{proof} We have
{\small 
\begin{eqnarray*}
(n_{58}',n_{59}',n_{60}') & = & (4615547228454,7504199621406,12197944701688) \\
(\F_{58},\F_{59},\F_{60}) & = & (591286729879, 956722026041,1548008755920) \\
(\T_{58},\T_{59},\T_{60}) & = & (752145307699165,1383410902447554,2544489349890656). 
\end{eqnarray*}
}
The stated inequalities hold for $g=58,59,60$ and hence for all $g \ge 58$ by Theorem~\ref{theorem main} and the recurrence relations of the Fibonacci and the tribonacci numbers. In fact, for the lower bound we have the slightly better estimate
\begin{equation}\label{better}
n'_g \ge 7.84\F_g
\end{equation} 
for all $g \ge 59$, since it holds for $g=59,60$.
\end{proof}

\begin{corollary} For all $g \ge 59$, we have
\begin{equation}\label{still better}
n_g \, \ge \, 7.84 \F_g.
\end{equation}
\end{corollary}
\begin{proof}
Follows from the inequality $n_g \ge n'_g$ and \eqref{better}.
\end{proof}
As far as we know, inequality \eqref{still better} is the best currently available lower bound on $n_g$ for large $g$.

\section{The gapsets graph}\label{section graph}

In this section, we show that the tree $\mathcal{T}$ is naturally embedded in a larger graph, which is easy to describe in terms of gapsets, or gapset filtrations, and which was actually discovered in this language.

\begin{definition} Let $F,F'$ be gapset filtrations of genus $g,g'$ respectively, where
$F=(F_0,\dots, F_{q-1})$, $F'=(F_0',\dots, F_{q'-1}')$. We put an edge between $F,F'$ if
\begin{itemize}
\item $g'=g+1$
\item $F'$ is a refinement of $F$, i.e. if $F_i' \supseteq F_i$ for all $i$.
\end{itemize}
\end{definition}

Clearly, all edges of the original tree $\mathcal{T}$ remain edges in the above new sense. But now new edges appear. In the figure below, the new edges are the dotted ones.  


\hspace{-0.9cm}
\begin{tikzpicture}[scale=1.4]

\coordinate (A) at (4.5,4);

\coordinate (B) at (4.5,3);

\coordinate (C1) at (3,2);
\coordinate (C2) at (6,2);

\coordinate (D1) at (1.5,1);
\coordinate (D2) at (4,1);
\coordinate (D3) at (5.5,1);
\coordinate (D4) at (7,1);

\coordinate (E1) at (-0.4,0);
\coordinate (E2) at (0.7,0);
\coordinate (E3) at (1.9,0);
\coordinate (E4) at (3.1,0);
\coordinate (E5) at (4.3,0);
\coordinate (E6) at (5.7,0);
\coordinate (E7) at (7.5,0);

\draw [fill] (A) circle (0.06) node [right] {$\emptyset$};

\draw [fill] (B) circle (0.06) node [right] {\,$(1)$};

	\draw [thick] (A)--(B);

\draw [fill] (C1) circle (0.06) node [left] {$(12)$\,\,};
\draw [fill] (C2) circle (0.06) node [right] {\,$(1)(1)$};

	\draw [thick] (B)--(C1);
	\draw [thick] (B)--(C2);

\draw [fill] (D1) circle (0.06) node [left] {$(123)$\,\,};
\draw [fill] (D2) circle (0.06) node [left] {$(12)(1)$\,};
\draw [fill] (D3) circle (0.06) node [right] {$(12)(2)$};
\draw [fill] (D4) circle (0.06) node [right] {$(1)(1)(1)$};

	\draw [thick] (C1)--(D1);
	\draw [thick] (C1)--(D2);
	\draw [thick] (C1)--(D3);
	\draw [thick] (C2)--(D4);
	
\draw [fill] (E1) circle (0.06) node [below] {$(1234)$};
\draw [fill] (E2) circle (0.06) node [below] {$(123)(1)$};
\draw [fill] (E3) circle (0.06) node [below] {$(123)(2)$};
\draw [fill] (E4) circle (0.06) node [below] {$(123)(3)$};
\draw [fill] (E5) circle (0.06) node [below] {$(12)(12)$};
\draw [fill] (E6) circle (0.06) node [below] {$(12)(1)(1)$};
\draw [fill] (E7) circle (0.06) node [below] {$(1)(1)(1)(1)$};

	\draw [thick] (D1)--(E1);
	\draw [thick] (D1)--(E2);
	\draw [thick] (D1)--(E3);
	\draw [thick] (D1)--(E4);
	\draw [thick] (D2)--(E5);
	\draw [thick] (D2)--(E6);
	\draw [thick] (D4)--(E7);

\draw [dashed] (C2)--(D2);
\draw [dashed] (D2)--(E2);
\draw [dashed] (D3)--(E3);
\draw [dashed] (D3)--(E5);
\draw [dashed] (D4)--(E6);

\end{tikzpicture}

\begin{remark}
While the original tree $\mathcal{T}$ has many leaves, i.e. vertices of degree $1$, this is no longer the case in our graph: every vertex has at least one child, as easily seen.
\end{remark}

Note also that the injective maps described in preceding sections use the new edges of this graph, not those of its subtree $\mathcal{T}$.

\section{Going further}\label{section further}

Here we characterize gapset filtrations of multiplicity $3$ and $4$, respectively. The proofs will appear in a subsequent paper using more tools. We also formulate two conjectures both implying $n_{g+1} \ge n_{g}$ for all $g \ge 0$.

\subsection{The case $m=3$}
\begin{theorem}\label{thm m=3} Let $a \ge 1, b \ge 0$. Then 
\begin{itemize}
\item $(12)^a(1)^b$ is a gapset filtration if and only if $b \le a+1$
\item $(12)^a(2)^b$ is a gapset filtration if and only if $b \le a$.
\end{itemize}
\end{theorem}

\begin{corollary}\label{cor m=3} For all $g \ge 1$, there is an explicit injection
$$\fc(g, m=3) \longrightarrow \fc(g+1, m=3).$$
\end{corollary}

\subsection{The case $m=4$}

\begin{theorem}\label{thm m=4} Let $a \ge 1, b,c \ge 0$. Then 
\begin{itemize}
\item $(123)^a(12)^b(1)^c$ is a gapset filtration $\iff$ $b,c \le a+1$
\item $(123)^a(12)^b(2)^c$ is a gapset filtration $\iff$ $b+c \le a+1, \,c \le a+b$ 
\item $(123)^a(13)^b(1)^c$ is a gapset filtration $\iff$ $c \le a+1$ 
\item $(123)^a(13)^b(3)^c$ is a gapset filtration $\iff$ $c \le a$ 
\item $(123)^a(23)^b(2)^c$ is a gapset filtration $\iff$ $b+c \le a$ 
\item $(123)^a(23)^b(3)^c$ is a gapset filtration $\iff$ $b,c \le a$.
\end{itemize}
\end{theorem}

\begin{corollary}\label{cor m=4} For all $g \ge 1$, there is an explicit injection
$$\fc(g, m=4) \longrightarrow \fc(g+1, m=4).$$
\end{corollary}

Here is a hopefully temporary paradox. The subtrees of $\mathcal{T}$ for $m=3$ and for $m=4$ grow pretty slowly, but proving their growth via Theorems~\ref{thm m=3} and \ref{thm m=4} and their corollaries is relatively easy and will be done in a forthcoming paper. On the other hand, computations show that the larger $m$ is, the more vigorous the growth of the corresponding subtree is. However, \emph{proving} that growth is still an open problem.

\subsection{Two conjectures}

We conclude this paper with two conjectures which would both imply Conjecture~\ref{simple growth}, namely $n_{g+1} \ge n_{g}$ for all $g \ge 0$. The first one would further confirm that `most' numerical semigroups are of depth $q \le 3$.

\begin{conjecture}\label{n'/n} One has $n_{g+1}' \ge 1.38\, n_g$\, for all $g \ge 1$, and $n_{g+1}' \ge 1.5\, n_g$ for all $g \ge 49$.
\end{conjecture}

The available data, namely the values of $n'_g$ for $1 \le g \le 60$ given in the Appendix and the values of $n_g$ given in~\cite{FH}, show that \emph{Conjecture~\ref{n'/n} holds for all $1 \le g \le 59$.}

\smallskip
Indeed, for $1 \le g \le 59$, the minimum of $n_{g+1}' / n_g$ is found to be attained at $g=18$, for which we have $n_{19}'/n_{18} \, \sim \, 1.3806341$. For $1 \le g \le 5$, the values of $n_{g+1}' / n_g$  are
$$
2,\, 2,\, 1.5,\, 1+4/7,\, 1+2/3,
$$
respectively, and yield $n_{g+1}' / n_g \ge 1.5$ in this range. A graphical representation of $n_{g+1}' / n_g$ in the range $6 \le g \le 59$ is given in Figure~\ref{quo}. The available data shows that $n_{g+1}' / n_g \ge 1.5$ holds for all $49 \le g \le 59$, and most probably beyond as well.

\begin{center}
\begin{figure}\label{quo}
\includegraphics[width=12cm]{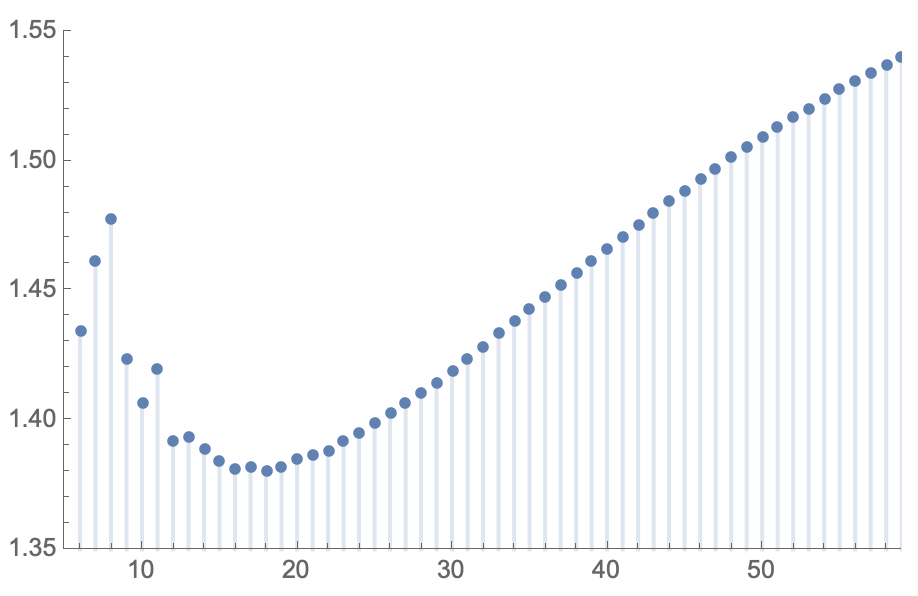}
\caption{The quotient $n_{g+1}' / n_g$ for $6 \le g \le 59$}
\label{quo}
\end{figure}
\end{center}

Our second conjecture states that the growth of the number of vertices of given genus should hold for many infinite subtrees of $\mathcal{T}$.

\begin{conjecture} Let $S$ be a numerical semigroup such that the subtree $\mathcal{T}(S)$ of $\mathcal{T}$ rooted at $S$ is infinite. Then the successive levels of $\mathcal{T}(S)$ have non-decreasing cardinalities.
\end{conjecture}

An interesting particular case is that of $S_m = \{0\} \cup [m,\infty[$. The conjecture seems to hold for $S_m$ and any $m \ge 2$. Since every numerical semigroup $S \not= \N$ is a descendant of some $S_m$ with $m \ge 2$, the validity of the above conjecture for $S_m$ for all $m \ge 2$ would imply the  conjecture $n_{g} \ge n_{g-1}$ for all $g \ge 1$. The cases $m=3,4,5$ have been established in~\cite{GMR} with computational methods. Our characterization above for $m=3,4$ yields a simpler proof in these two cases. Finally, it would be very interesting to determine the asymptotic growth rate of these particular subtrees $\mathcal{T}(S_m)$.

\section{Appendix}
Here is the sequence of $n'_g$ for $g=1,\dots,60$, computed using the fast algorithms developped in~\cite{FH} and made on CALCULCO, the high performance computing platform of our university~\cite{cal}.

{\small
\begin{figure}[h]
\begin{quote}
1, 2, 4, 6, 11, 20, 33, 57, 99, 168, 287, 487, 824, 1395, 2351, 3954, 6636, 11116, 18593, 31042, 51780, 86223, 143317, 237936, 394532, 653420, 1080981, 1786328, 2948836, 4863266, 8013802, 13194529, 21707242, 35684639, 58618136, 96221845, 157840886, 258749944, 423906805, 694076610, 1135816798, 1857750672, 3037078893, 4962738376, 8105674930, 13233250642, 21595419304, 35227607540, 57443335681, 93635242237, 152577300884, 248541429293, 404736945777, 658898299876, 1072361202701, 1744802234628, 2838171714880, 4615547228454, 7504199621406, 12197944701688.
\end{quote}
\caption{The sequence of $n'_g$ for $g=1,\dots,60$}
\label{seq n'g}
\end{figure}
}


\end{document}